\documentclass[11pt]{article}
\usepackage[a4paper,margin=25mm]{geometry}
\usepackage[T1]{fontenc}
\usepackage{lmodern}
\usepackage{amsmath,amssymb,amsthm,mathtools}
\usepackage{enumitem}
\usepackage[hidelinks]{hyperref}
\usepackage{microtype}

\newtheorem{theorem}{Theorem}[section]

\theoremstyle{definition}
\newtheorem{remark}[theorem]{Remark}

\newcommand{\pres}[2]{\left\langle #1\ \middle|\ #2\right\rangle}
\newcommand{\normal}{\mathrel{\unlhd}}

\title{\textbf{Two Results on Asphericity}}
\author{Roman Mikhailov}
\date{}

\begin{document}
\maketitle

\begin{abstract}
We construct an explicit finite chain of non-aspherical group presentation
complexes
\(K(\mathcal X)\subset K(\mathcal Y)\subset K(\mathcal Z)\) in which
\(K(\mathcal Y)\) is Cockcroft, both inclusion-induced maps on
\(\pi_2\) are zero, and the pair
\((K(\mathcal Y),K(\mathcal X))\) has the identity property.
This answers a question of Hitchman.  The construction is given
directly from a short balanced presentation of the binary icosahedral
group.  As the second result, we construct
an integral Lie ring presentation in
which a subpresentation has a nontrivial identity among
relations, whereas the presentation itself has none. 
\end{abstract}

\section*{Introduction}

For a connected two-dimensional CW-complex \(K\), asphericity is
equivalent to \(\pi_2(K)=0\).  Whitehead's asphericity question asks
whether every subcomplex of an aspherical two-complex is aspherical.
This question remains open. For the results and approaches related to this question see \cite{HogAngeloni1993}, \cite{Ros07}.

A natural approach to finding a non-aspherical subcomplex of an
aspherical one (see \cite{Dyer} and \cite{Hitchman}) is to construct
an infinite chain of two-dimensional complexes 
\[
 X_1\subset X_2\subset\cdots\subset X_i\subset X_{i+1}\subset\cdots
 \tag{0.1}
\]
such that
\[
 \pi_2(X_1)\ne0,\qquad
 \bigl(\pi_2(X_i)\longrightarrow\pi_2(X_{i+1})\bigr)=0
 \quad(i\ge1).
 \tag{0.2}
\]
If \(X_\infty=\bigcup_iX_i\), every map \(S^2\to X_\infty\) has image
in some \(X_i\) and becomes null-homotopic in \(X_{i+1}\).  Thus
\(\pi_2(X_\infty)=0\), while \(X_1\) is a non-aspherical subcomplex.
The union would be a counterexample to Whitehead's question.

Recall that a connected two-complex \(K\) is Cockcroft when its
ordinary Hurewicz homomorphism
\(\pi_2(K)\to H_2(K)\) is zero \cite{Cockcroft}.  This is exactly the
condition needed
for one step: if \(K\) is Cockcroft, attach two-cells along a
finite generating set of \(\pi_1K\).  The resulting complex \(L\) is
simply connected, and naturality of Hurewicz  map gives
\[
 \bigl(\pi_2(K)\longrightarrow\pi_2(L)\bigr)=0.
\]
Conversely, a zero map
\(\pi_2(K)\to\pi_2(L)\) forces \(K\) to be Cockcroft because
\(H_2(K)\to H_2(L)\) is injective for an inclusion of
two-complexes: the relative cellular chain group in dimension \(3\)
is zero.
This is the ordinary-cover case of \cite[Lemma~1]{BogleyDyer}.

The first nontrivial case is therefore a chain
\[
 K_1\subset K_2\subset K_3,\qquad
 \pi_2(K_1)\ne0,\qquad
 \bigl(\pi_2(K_1)\to\pi_2(K_2)\bigr)=0,\qquad
 \bigl(\pi_2(K_2)\to\pi_2(K_3)\bigr)=0.
 \tag{0.3}
\]
The middle complex \(K_2\) must be Cockcroft.  The construction
above gives no reason for this.  We give a
single explicit eight-generator construction, based on \(SL(2,5)\),
for which \(K_1\) and \(K_2\) are non-aspherical, \(K_2\) is
Cockcroft, and both maps in \textup{(0.3)} are zero; see
Theorem~\ref{thm:icosahedral-chain}.  The last complex is not
Cockcroft, so this construction does not give a third zero map.

This result leads to the following natural problem: how long can a
chain of inclusions of two-complexes be if the first complex is
non-aspherical and every inclusion induces the zero map on \(\pi_2\)?
In particular, at present the author does not know how to construct
\[
 X_1\subset X_2\subset X_3\subset X_4
\]
with \(\pi_2(X_1)\ne0\) and all three inclusion-induced maps on
\(\pi_2\) equal to zero.

In \cite{Dyer}, Dyer replaced the homotopical property by a homological one and suggested a strategy for finding an infinite chain \textup{(0.1)}. The additional condition used in Dyer's program is the identity
property.  We recall its precise form from \cite{Hitchman}. Consider group presentations and their standard two-complexes
\[
 \mathcal P=\langle\mathbf x\mid\mathbf r\rangle,\qquad
 \mathcal Q=\langle\mathbf x\mid\mathbf r,\mathbf s\rangle,
 \qquad K(\mathcal P)=X\subset Y=K(\mathcal Q),
 \tag{0.4}
\]
and put \(F=F(\mathbf x)\) and \(H=\pi_1(Y)\).  For \(N\le H\), let
\(N_F\) be its inverse image in \(F\).  In a spherical picture over
\(\mathcal Q\), choose transverse paths from the global basepoint to
the \(\mathbf s\)-disks.  If such a disk has label
\(s_i^{\varepsilon_i}\), where \(\varepsilon_i=\pm1\), let
\(\omega_i\in F\) be the word determined by its transverse path.
The pair \((Y,X)\) has the right \(N\)-identity property if, in every
spherical picture, the \(\mathbf s\)-disks can be paired
\(i\leftrightarrow j\) so that
\[
 s_i=s_j,\qquad \varepsilon_i=-\varepsilon_j,\qquad
 N_F\omega_i=N_F\omega_j.
 \tag{0.5}
\]
In the left version, the last equality in (0.5) is replaced by
\(\omega_iN_F=\omega_jN_F\).  When \(N\) is normal the two versions
agree; the case \(N=\{1\}\) is called the identity property
\cite[Definition~3.1]{Hitchman}.

There is an equivalent cellular formulation.  Write
\[
 C_2(\widetilde Y)
   =\mathbb Z[H]\,\mathbf r\oplus\mathbb Z[H]\,\mathbf s
\]
and let
\[
 j:\pi_2(Y)\longrightarrow\mathbb Z[H]\,\mathbf s
 \tag{0.6}
\]
be the projection of a spherical cellular cycle onto the coordinates
of the new two-cells.  If \(I_N\) is the augmentation ideal of
\(\mathbb Z[N]\), Proposition~3.3 of \cite{Hitchman} says that the right
\(N\)-identity property is equivalent to all coefficients of
\(j(\pi_2Y)\) lying in \(I_N\mathbb Z[H]\).  Hence
\[
 (Y,X)\text{ has the identity property}
 \quad\Longleftrightarrow\quad j=0,
 \tag{0.7}
\]
because \(I_{\{1\}}=0\).

Hitchman's Example~4.5 gives a non-aspherical pair \(X\subset Y\)
with the identity property and with \(\pi_2(X)\to\pi_2(Y)\) zero, but
\(Y\) is not Cockcroft and therefore cannot be the middle term of a
second zero step.  Immediately after that example Hitchman writes:
\begin{quote}
``It would be of considerable interest to find non trivial examples''
\end{quote}
with \(\pi_2(X)\ne0\) satisfying all three conditions
\cite[p.~309, immediately before Corollary~4.6]{Hitchman}:
\[
 \bigl(\pi_2(X)\longrightarrow\pi_2(Y)\bigr)=0,\qquad
 (Y,X)\text{ has the identity property},\qquad
 Y\text{ is Cockcroft}.
 \tag{0.8}
\]
Theorem~\ref{thm:icosahedral-chain} answers this question.  It gives
more than the requested
finite pair. Observe that, the construction given here is flexible and can be generalized to other presentations of superperfect groups. 

At the level of individual complexes, asphericity can be formulated
entirely in terms of group presentations.  Indeed, let \(K\) be a
connected finite two-dimensional CW-complex.  Choose a maximal tree
in its \(1\)-skeleton and collapse it.  Since the tree is a
contractible subcomplex, this quotient is homotopy equivalent to
\(K\).  It has one \(0\)-cell, and its oriented \(1\)-cells and the
attaching words of its \(2\)-cells give a finite presentation
\(\mathcal P\) of \(\pi_1K\).  Thus
\[
 K\simeq K(\mathcal P);
\]
see \cite[Section~2.3, p.~28]{Hillman}.  In particular, no additional
wedge of \(2\)-spheres is needed.  For an inclusion of complexes,
preserving the inclusion as a subpresentation is an extra
requirement; our constructions have this property by definition. 

This presentation viewpoint has a natural Lie-algebra analogue.
Ellis introduced the second homotopy module
\(\pi_2(\mathcal P)\) of a Lie-algebra presentation
\(\mathcal P\) in \cite{Ellis}.  We call such a presentation
aspherical when this module vanishes.  The analogous question asks
whether every subpresentation of an aspherical Lie-algebra
presentation is aspherical.

Theorem~\ref{thm:lie-subpresentation} gives an integral example: the
smaller presentation has an explicit nonzero identity among
relations, the larger one has none, both conclusions survive base
change to every field.
Thus the theorem gives a negative answer to the
Lie-algebra subpresentation question over every field.

The author thanks L. Bartholdi, V. Ionin and A. Semidetnov for discussions related to the subject of the paper. 

\paragraph{AI use statement.}
This text was written with the assistance of ChatGPT 5.6 Sol. Working with AI was an iterative process, making it difficult to isolate its specific contribution. The constructions presented in this paper are relatively simple and flexible and can be easily generalized.

\section{An explicit two-step chain}
We will repeatedly use the following elementary cellular
calculation.  For a finite presentation
\[
 \mathcal P=\pres{g_1,\ldots,g_n}{r_1,\ldots,r_m},
\]
the complex \(K(\mathcal P)\) has one \(0\)-cell, one oriented
\(1\)-cell for each \(g_j\), and one \(2\)-cell for each \(r_i\).
Consequently its cellular chain complex over \(\mathbb Z\) is
\[
 0\longrightarrow\mathbb Z^m
 \xrightarrow{\ \partial_2\ }\mathbb Z^n
 \xrightarrow{\ 0\ }\mathbb Z\longrightarrow0,
 \qquad
 (\partial_2)_{ji}=\operatorname{exp}_{g_j}(r_i),
\]
where \(\operatorname{exp}_{g_j}(r_i)\) is the total exponent of
\(g_j\) in the word \(r_i\).  Indeed, collapse all \(1\)-cells except
the \(g_j\)-cell.  The attaching map of the \(r_i\)-cell then becomes
a map \(S^1\to S^1\), and its degree is precisely this exponent sum.
This is the cellular boundary formula
\cite[Section~2.2, pp.~140--141]{Hatcher}; see in particular
\cite[Example~2.36]{Hatcher}.  Hence
\[
 H_2K(\mathcal P)=\ker\partial_2,\qquad
 H_1K(\mathcal P)=\operatorname{coker}\partial_2.
\]
In general these two groups can be read from the Smith normal form
of the integer matrix \(\partial_2\).  In both applications below,
the kernel is visible directly, so no reduction is needed.
Equivalently,
\[
 \operatorname{exp}_{g_j}(r_i)
 =
 \varepsilon\!\left(\frac{\partial r_i}{\partial g_j}\right),
\]
where the expression on the right is the augmentation of the Fox
derivative \cite{Fox}.  This computes the ordinary homology of the
presentation complex.  The cellular boundary in the universal cover
is instead the full Fox matrix over the group ring; the exponent-sum
matrix alone does not compute \(\pi_2\).

The balanced presentation
\[
 \mathcal P=\pres{x,y}{x^2yx^{-1}y,\;xy^4xy^{-1}}
 \tag{1.1}
\]
defines \(SL(2,5)\), the binary icosahedral group of order \(120\)
\cite{HavasRamsay}.  Its exponent-sum matrix is
\[
 \begin{pmatrix}1&2\\2&3\end{pmatrix},
 \qquad \det=-1.
 \tag{1.2}
\]

Put
\[
 B=\pres{a,b,c,d}{[a,c],[a,d],[b,c],[b,d]}
   \cong F(a,b)\times F(c,d).
 \tag{1.3}
\]
Starting directly with the free product of the presentation
\textup{(1.1)} and \(B\), form two HNN extensions, identifying
\(\langle x[a,b]\rangle\) with \(\langle [c,d]\rangle\) and
\(\langle y[a,b]\rangle\) with \(\langle [a,b]^2\rangle\).  Thus
\[
\begin{aligned}
\mathcal X=\langle
&x,y,a,b,c,d,\tau,\sigma\mid
 x^2yx^{-1}y,\;xy^4xy^{-1},\\
&[a,c],[a,d],[b,c],[b,d],\\
&\tau^{-1}x[a,b]\tau[c,d]^{-1},\\
&\sigma^{-1}y[a,b]\sigma[a,b]^{-2}
\rangle .
\end{aligned}
\tag{1.4}
\]
Thus \(\mathcal X\) has eight generators and eight relators; no
auxiliary generator is needed.

Now attach two-cells along \(\tau\) and \(\sigma\), and then attach
four more two-cells along \(a,b,c,d\).  In presentation notation,
\[
 \mathcal Y=\langle\mathcal X\mid\tau,\sigma\rangle,
 \qquad
 \mathcal Z=\langle\mathcal Y\mid a,b,c,d\rangle.
 \tag{1.5}
\]
No generators are added at these two stages, so this gives a 
chain of finite two-complexes
\[
 K(\mathcal X)\subset K(\mathcal Y)\subset K(\mathcal Z).
 \tag{1.6}
\]

\begin{theorem}\label{thm:icosahedral-chain}
The complex \(K(\mathcal X)\) is non-aspherical, and
\(K(\mathcal Y)\) is non-aspherical and Cockcroft.  Moreover,
\[
 \pi_2K(\mathcal X)\longrightarrow\pi_2K(\mathcal Y)
 \quad\text{and}\quad
 \pi_2K(\mathcal Y)\longrightarrow\pi_2K(\mathcal Z)
 \tag{1.7}
\]
are both zero.  The pair
\((K(\mathcal Y),K(\mathcal X))\) has the identity property.
\end{theorem}

\begin{proof}
We verify the assertions in six short steps.

\smallskip
\noindent\textit{1.  The base group embeds.}
Let \(S=SL(2,5)\), presented by \textup{(1.1)}.  The standard
complex of \(B\) is the product of two bouquets, so it is
aspherical, and \([a,b],[c,d],[a,b]^2\) have infinite order.  Moreover
\(x,y\ne1\) in \(S\) (otherwise the two relations force \(S=1\)).
Hence \(x[a,b]\) and \(y[a,b]\) are reduced words
of length two in \(S*B\), and therefore also have infinite order.

Thus \(G(\mathcal X)\) is the multiple HNN extension of \(S*B\)
with stable letters \(\tau,\sigma\) and associated cyclic subgroups
\[
 \langle x[a,b]\rangle\xrightarrow{\cong}\langle [c,d]\rangle,
 \qquad
 \langle y[a,b]\rangle\xrightarrow{\cong}\langle [a,b]^2\rangle.
\]
Britton's lemma shows that the base group \(S*B\) embeds in
\(G(\mathcal X)\) \cite[I.5.2, Theorem~11]{Serre}.

\smallskip
\noindent\textit{2.  The source is non-aspherical.}
Since the matrix \textup{(1.2)} is unimodular, the presentation
complex \(K(\mathcal P)\) of \textup{(1.1)} satisfies
\[
 H_1K(\mathcal P)=H_2K(\mathcal P)=0.
\]
By Step~1, \(S\) embeds in \(G(\mathcal X)\).  Thus the latter group
has torsion, and
therefore \(K(\mathcal X)\) is not aspherical: the fundamental group
of a finite-dimensional aspherical complex is torsion-free
\cite[Chapter~VIII, Section~2]{Brown}.

We also record the part of \(\pi_2K(\mathcal X)\) that will be used
below.  Put \(G=G(\mathcal X)\) and compute the Fox boundary over
\(\mathbb Z[G]\), using row vectors with coefficients on the left.  In
the \(\tau\)-column only the relator
\[
R_\tau=\tau^{-1}x[a,b]\tau[c,d]^{-1}
\]
contributes, and
\[
 \frac{\partial R_\tau}{\partial\tau}=\tau^{-1}(x[a,b]-1).
\]
Right multiplication by \(h-1\) in \(\mathbb Z[G]\) is injective
whenever \(h\) has infinite order: a nonzero element with finite
support cannot be invariant under translation by \(h\).  Since
\(x[a,b]\) has infinite order, the coefficient of \(R_\tau\) in every
Fox cycle is zero.  The \(\sigma\)-column gives in the same way
\[
 \frac{\partial(\sigma^{-1}y[a,b]\sigma[a,b]^{-2})}{\partial\sigma}
 =\sigma^{-1}(y[a,b]-1),
\]
so the coefficient of the second HNN relator is also zero.

The remaining equations split into the \(S\)-block and the
\(B\)-block.  The latter has zero kernel because \(K(B)\) is
aspherical.  This stays true over \(\mathbb Z[G]\): choosing left-coset
representatives makes \(\mathbb Z[G]\) a direct sum of copies of
\(\mathbb Z[B]\) as a right \(\mathbb Z[B]\)-module.  Applying the same
argument to \(S\le G\) shows that the kernel of the \(S\)-block is
generated over \(\mathbb Z[G]\) by the kernel for \(\mathcal P\).
Consequently
\[
 \pi_2K(\mathcal X)
 =\mathbb Z[G]\cdot
 \operatorname{im}\bigl(\pi_2K(\mathcal P)\longrightarrow
                         \pi_2K(\mathcal X)\bigr).
\]

\smallskip
\noindent\textit{3.  The first map in (1.7) is zero.}
In \(G(\mathcal Y)\), the new relators give \(\tau=\sigma=1\).
The two HNN relations in \textup{(1.4)} then give
\[
 x=[c,d][a,b]^{-1},\qquad y=[a,b].
 \tag{1.8}
\]
Since \([a,b]\) and \([c,d]\) commute, the two icosahedral relators reduce to
\[
 [a,b][c,d]=1,\qquad [a,b][c,d]^2=1.
 \tag{1.9}
\]
Hence \([a,b]=[c,d]=1\), and therefore \(x=y=1\) in \(G(\mathcal Y)\).

Thus the map from the icosahedral presentation complex (1.1) to
\(K(\mathcal Y)\) is trivial on fundamental groups and lifts to the
universal cover of \(K(\mathcal Y)\).  The Hurewicz image of any
element of its \(\pi_2\) is therefore the image of a class in the
ordinary \(H_2\) of the icosahedral presentation complex.  That group
is zero by Step~2.  Since the universal cover of
\(K(\mathcal Y)\) is simply connected, its Hurewicz map in dimension
two is an isomorphism.  Hence the icosahedral \(\pi_2\) maps to zero.
The generating statement in Step~2 now gives
\[
 \bigl(\pi_2K(\mathcal X)\longrightarrow
 \pi_2K(\mathcal Y)\bigr)=0.
 \tag{1.10}
\]

\smallskip
\noindent\textit{4.  The target is Cockcroft and non-aspherical.}
Equations (1.8)--(1.9) also show directly that
\[
 G(\mathcal Y)
 \cong
 \pres{a,b,c,d}
 {[a,c],[a,d],[b,c],[b,d],[a,b],[c,d]}
 \cong\mathbb Z^4.
 \tag{1.11}
\]
We compute \(H_2\) directly from the original presentation
\(\mathcal Y\), which has eight generators and ten relators.  Every
relator has exponent sum zero in \(a,b,c,d\).  The two HNN
relators have exponent sum one in \(x\) and \(y\), respectively, and
zero in \(\tau,\sigma\).  The two new relators \(\tau,\sigma\) have
exponent sum one in their respective generators.  Thus the four
columns belonging to the two HNN relators and to
\(\tau,\sigma\) form the identity matrix in the rows
\(x,y,\tau,\sigma\).  All columns are zero in the other four rows.
Consequently the image of
\(\partial_2:\mathbb Z^{10}\to\mathbb Z^8\) is exactly the direct
summand spanned by \(x,y,\tau,\sigma\), and its kernel is free of
rank \(10-4=6\).  Therefore
\[
 H_2K(\mathcal Y)=\ker\partial_2=\mathbb Z^6.
 \tag{1.12}
\]
The surjection in Hopf's exact sequence
\[
 H_2K(\mathcal Y)\longrightarrow
 H_2(\mathbb Z^4)\cong\bigwedge\nolimits^2\mathbb Z^4
 \cong\mathbb Z^6
\]
has source and target free abelian of rank six, so it is an
isomorphism.  In the exact sequence
\[
 \pi_2K(\mathcal Y)\longrightarrow H_2K(\mathcal Y)
 \longrightarrow H_2(\mathbb Z^4)\longrightarrow0,
\]
the first arrow is the ordinary Hurewicz map and its image is the
kernel of the second arrow.  This kernel is zero, proving that
\(K(\mathcal Y)\) is Cockcroft.

If \(K(\mathcal Y)\) were aspherical, it would be a
two-dimensional \(K(\mathbb Z^4,1)\).  This is impossible because
\[
 H_3(\mathbb Z^4)\cong\bigwedge\nolimits^3\mathbb Z^4\ne0.
 \]
Thus \(K(\mathcal Y)\) is non-aspherical.

\smallskip
\noindent\textit{5.  The pair has the identity property.}
The generator \(\tau\) occurs in only one old relator,
\[
 \tau^{-1}x[a,b]\tau[c,d]^{-1},
\]
and
\[
\frac{\partial\bigl(
 \tau^{-1}x[a,b]\tau[c,d]^{-1}\bigr)}{\partial\tau}
 =\tau^{-1}\bigl(x[a,b]-1\bigr).
\]
In \(G(\mathcal Y)\) we have
\(x=y=[a,b]=[c,d]=\tau=\sigma=1\), whereas the Fox derivative of
the new relator \(\tau\) is \(1\).  Hence the \(\tau\)-coordinate of the
cellular boundary of a spherical cycle forces the coefficient of
the new \(\tau\)-cell to be zero.  The identical argument with
\[
 \sigma^{-1}y[a,b]\sigma[a,b]^{-2}
\]
forces the coefficient of the new \(\sigma\)-cell to be zero.  Thus
the projection of
\(\pi_2K(\mathcal Y)\) onto the two new cell coordinates vanishes.
By \cite[Proposition~3.3]{Hitchman}, this is exactly the identity
property.

\smallskip
\noindent\textit{6.  The second map in (1.7) is zero.}
By (1.11), adjoining the four relators \(a,b,c,d\) makes
\(K(\mathcal Z)\) simply connected.  By naturality and the Cockcroft
property of \(K(\mathcal Y)\), every element coming from
\(\pi_2K(\mathcal Y)\) has zero Hurewicz image in
\(H_2K(\mathcal Z)\).  Since the Hurewicz map for the simply
connected complex \(K(\mathcal Z)\) is an isomorphism, every such
element is zero in \(\pi_2K(\mathcal Z)\).  This proves the second
map in (1.7) is zero.
\end{proof}

\begin{remark}
The chain stops at this point.  Indeed,
\(\chi(K(\mathcal Z))=1-8+14=7\), and \(K(\mathcal Z)\) is simply
connected.  Hence
\[
 \pi_2K(\mathcal Z)\cong H_2K(\mathcal Z)\cong\mathbb Z^6,
\]
so \(K(\mathcal Z)\) is not Cockcroft.
\end{remark}

\section{An integral Lie ring example}

Let \(F=\operatorname{Lie}_{\mathbb Z}(X)\) be the free Lie ring on a
finite set \(X\), let \(R\normal F\) be the ideal generated by 
relations \(\rho_1,\ldots,\rho_m\), and set \(G=F/R\).  The abelianized
relation ideal
\[
 R^{\mathrm{ab}}=R/[R,R]
\]
is a left module over the universal enveloping ring \(U(G)\).  The
relations define
\[
 \partial_{\mathcal P}:
 \bigoplus_{i=1}^m U(G)e_{\rho_i}\longrightarrow R^{\mathrm{ab}},
 \qquad
 e_{\rho_i}\longmapsto \rho_i+[R,R].
 \tag{2.1}
\]
We write
\[
 \Pi_2(\mathcal P)=\ker\partial_{\mathcal P}
 \tag{2.2}
\]
for the integral module of identities among relations.  Over a field,
the corresponding kernel is the second homotopy module
\(\pi_2(\mathcal P)\) of the presentation introduced in
\cite[Theorem~17]{Ellis}; the presentation is called aspherical when
this module vanishes.

In the free Lie ring on \(x,y,a,b,c\), define
\[
 r_1=x+[x,[x,y]],\qquad
 r_2=y-[y,[x,y]].
 \tag{2.3}
\]
Let
\[
 \mathcal L_1=\pres{x,y}{r_1,r_2}.
 \tag{2.4}
\]
Adjoin \(a,b,c\) and the relations
\[
\begin{aligned}
 s_1&=a+[[x,y],a]+[y,c]-[x,y],\\
 s_2&=b-[[x,y],b]-[x,c],\\
 s_3&=c+[x,a]-[y,b],
\end{aligned}
\tag{2.5}
\]
and set
\[
\mathcal L_2=
 \pres{x,y,a,b,c}{r_1,r_2,s_1,s_2,s_3}.
 \tag{2.6}
\]
We have the inclusion \(\mathcal L_1\subset\mathcal L_2\) of presentations of Lie rings. 

\begin{theorem}\label{thm:lie-subpresentation}
For the integral Lie ring presentations
\(\mathcal L_1\subset\mathcal L_2\):
\[
 \Pi_2(\mathcal L_1)\neq0,\qquad
 \Pi_2(\mathcal L_2)=0.
 \tag{2.7}
\]
After base change to every field \(k\), the smaller presentation
remains non-aspherical and the larger one is aspherical.
Consequently \(\mathcal L_1\subset\mathcal L_2\) is a counterexample
to the Lie algebra version of Whitehead's asphericity question.
In particular, this inclusion kills the explicit nonzero identity
below.
\end{theorem}

\begin{proof}
 There is an exact free Lie identity in the free Lie ring:
\[
 [y,r_1]+[x,r_2]=0.
 \tag{2.9}
\]
Indeed, after cancelling \([y,x]+[x,y]\), equation (2.9) is the
Jacobi identity for \(x,y,[x,y]\).  Therefore
\[
 \omega= y\,e_{r_1}+x\,e_{r_2}
 \tag{2.10}
\]
lies in \(\Pi_2(\mathcal L_1)\).  Let \(G_1\) be the Lie ring presented by \(\mathcal L_1\). The underlying abelian group of
\(G_1\) is free, so the integral PBW theorem embeds \(G_1\) into
\(U(G_1)\).  Both coefficients in (2.10) are nonzero, and the direct
sum in (2.1) is free.  Thus \(\omega\neq0\).  The same argument after
base change, using PBW over a field, shows that this identity remains
nonzero over every field.

For the larger presentation, the key point is the following
exact identity in the free Lie ring:
\[
x=[x,s_1]+[y,s_2]-[[x,y],s_3]+[a,r_1]+[b,r_2]+r_1 .
\tag{2.11}
\]
To verify it, expand the right side.  The linear pairs
\([x,a]+[a,x]\), \([y,b]+[b,y]\), and
\(-[x,[x,y]]+[x,[x,y]]\) cancel.  The remaining terms are the
three Jacobi sums
\[
\begin{aligned}
&[x,[[x,y],a]]-[[x,y],[x,a]]+[a,[x,[x,y]]],\\
&-[y,[[x,y],b]]+[[x,y],[y,b]]-[b,[y,[x,y]]],\\
&[x,[y,c]]-[y,[x,c]]-[[x,y],c],
\end{aligned}
\]
and hence vanish.

Let \(R_2\) be the ideal generated by the five marked relations in
\(\mathcal L_2\).  Equation (2.11) gives \(x\in R_2\), hence
\([x,y]\in R_2\).  The relation \(r_2\) then gives \(y\in R_2\);
successively \(s_1,s_2,s_3\) give \(a,b,c\in R_2\).  Thus
\[
 R_2=F,\qquad G_2=F/R_2=0.
 \tag{2.12}
\]
Consequently \(U(G_2)=\mathbb Z\) and
\[
 R_2^{\mathrm{ab}}=F/[F,F]
 \cong\mathbb Z\{x,y,a,b,c\}.
\]
Modulo \([F,F]\), the five marked relations
\((r_1,r_2,s_1,s_2,s_3)\) have respective classes
\((x,y,a,b,c)\).  Therefore (2.1) is the identity matrix
\[
 \partial_{\mathcal L_2}:\mathbb Z^5
 \xrightarrow{\ \cong\ }\mathbb Z^5,
\]
and \(\Pi_2(\mathcal L_2)=0\).  The same calculation survives every
base change \(\mathbb Z\to k\).  Thus over every field the second
homotopy module of the larger presentation vanishes, so that
presentation is aspherical.
\end{proof}

\bigskip
\noindent
Saint Petersburg State University\\
7/9 Universitetskaya nab., St.~Petersburg, 199034 Russia
\end{document}